\documentclass[12pt,oneside,reqno]{amsart}
\usepackage{geometry}                
\usepackage{graphicx}
\usepackage{amssymb}
\usepackage{tikz}
\usetikzlibrary{calc}
\usepackage{bm}
\usepackage{mathtools}
\usepackage{mathabx}

\usepackage{amsmath,amsthm,amscd,amssymb}
\usepackage{latexsym}
\usepackage[colorlinks,citecolor=red,pagebackref,hypertexnames=false]{hyperref}

\numberwithin{equation}{section}
\theoremstyle{plain}
\newtheorem{theorem}{Theorem}[section]
\newtheorem{lemma}[theorem]{Lemma}

\newtheorem{proposition}[theorem]{Proposition}
\newtheorem{conjecture}[theorem]{Conjecture}

\theoremstyle{definition}
\newtheorem{definition}[theorem]{Definition}

\theoremstyle{remark}
\newtheorem{remark}[theorem]{Remark}

\newtheorem{summary}[theorem]{Summary}

\newtheorem{case[theorem]}{Case}

\newcommand{\G}{G_{k+1,m}}

\begin{document}

\title{Finite point configurations in the plane, rigidity and Erd\H os problems} 

\author{A. Iosevich and J. Passant}

\date{May 10, 2018}

\email{iosevich@math.rochester.edu}
\email{jpassant@ur.rochester.edu}

\address{Department of Mathematics, University of Rochester, Rochester, NY 14627}

\thanks{This work was partially supported by the NSA Grant H98230-15-1-0319}

\begin{abstract} For a finite point set $E\subset \mathbb{R}^d$ and a connected graph $G$ on $k+1$ vertices, we define a $G$-framework to be a collection of $k + 1$ points in E such that the distance between a pair of points is specified if the corresponding vertices of $G$ are connected by an edge. We consider two frameworks the same if the specified edge-distances are the same.
We find tight bounds on such distinct-distance drawings for rigid graphs in the plane, deploying the celebrated result of Guth and Katz.
We introduce a congruence relation on the wider set of graphs, which behaves nicely in both the real-discrete and continuous settings. We provide a sharp bound on the number of such congruence classes.
We then make a conjecture that the tight bound on rigid graphs should apply to all graphs. This appears to be a hard problem even in the case of the non-rigid 2-chain. However we provide evidence to support the conjecture by demonstrating that if the Erd\H os pinned-distance conjecture holds in dimension $d$ then the result for all graphs in dimension $d$ follows.
\end{abstract} 

\maketitle


\section{Introduction}

\vskip.125in 

Given a set $E$ in $\mathbb{R}^d$, the distance set of $E$ is
$$\Delta_d(E) = \{|x - y| : x, y \in E\} \subseteq \mathbb{R}.$$
In \cite{E45} Erd\H os posed the question: What is the minimal number of distinct distances determined by a finite point set $E$ in $\mathbb{R}^d$?
This has been thoroughly studied in both the $d=2$ case where the cascade of improvements to Erd\H os original $|E|^\frac{1}{2}$ by authors including Moser \cite{M52}, Chung \cite{C84}, Chung-Szemer\' edi-Trotter \cite{CST92}, Sz\' ekely \cite{S97}, Solymosi-T\' oth \cite{ST01}, Tardos \cite{T03} and most recently the solution of the problem in two dimensions due to Guth-Katz \cite{GK15}. In higher dimensions a simple variant of Erd\H os original argument gives $|E|^\frac{1}{d}$ in dimension $d$. An improvement in three dimensions due Clarkson-Edelsbrunner-Gubias-Sharir-Welzl \cite{CEGSW90} proved that one obtains at least $|E|^{\frac{1}{2}}$ distances, the three dimentional bound was furthered by Aronov-Pach-Sharir-Tardos \cite{APST04} who also proved a small improvement over the $|E|^{\frac{1}{d}}$ bound in dimension $d$. This was then improved significantly by Solymosi-Vu (\cite{SV08}, see also \cite{SV04}) who proved one obtains at least $|E|^{\frac{2}{d} - \frac{2}{d(d+2)}}$ distances, a near optimal bound for large dimensions.

The study of distance sets may be viewed as the study of congruence classes of two-point configurations. If we consider a pair of points $x,y$ and another pair $x',y'$, then there exists a rigid motion $T$ such that $Tx=x', Ty=y'$ if and only if $|x-y|=|x'-y'|$. A similar question can be asked about configurations involving more points. in this paper we shall consider $(k+1)$-point configurations. Suppose that $k \leq d$ and let $x^1, x^2, \dots, x^{k+1}$ be linearly independent. Also assume that $y^1, y^2, \dots, y^{k+1}$ are linearly independent. Then the question of whether the two collections are congruent, i.e whether there exists a rigid motion $T$ such that $y^j=Tx^j$, $1 \leq j \leq k+1$ reduces to checking whether $|x^i-x^j|=|y^i-y^j|$ for all $1 \leq i<j \leq k+1$. 

\begin{center}
\begin{figure}[ht] 
\begin{tikzpicture}[scale=2]
\coordinate (A) at (0cm,0cm);
\coordinate (B) at (1.4cm,1cm);
\coordinate (C) at (2cm,0cm);
\coordinate (D) at (0.5cm,-0.7cm);
\draw[thick, line cap=round] (A) -- (B);
\draw[thick, line cap=round] (B) -- (C);
\draw[thick, line cap=round] (C) -- (D);
\draw[thick, line cap=round] (D) -- (A);
\draw[thick, line cap=round] (A) -- (C);
\draw[thick, dashed, line cap=round] (B) -- (D);
\end{tikzpicture}
\caption{$d=2, k=3$.}\label{figure1} 
\end{figure}
\end{center}

In the situation when $k>d$, significant new complications arise. As a simple example, consider Figure 1 above. The length of the dotted line is determined by the lengths of the solid lines, the natural dimension of the configuration space, in the sense that will be made precise, is $5$. In general, the following heuristic is extremely useful in understanding the situation. Each of the $k+1$ vectors has $d$ coordinates. The dimension of the Euclidean motion group in ${\Bbb R}^d$ is equal to $d$ plus the dimension of the orthogonal group. This yields 
$$ d(k+1)-d-{d \choose 2}=d(k+1)-{d+1 \choose 2}.$$ 

We now turn to precise definitions and statements of results. Given a finite set $E \subset {\Bbb R}^d$ of size $>k+1$, we consider $(k+1)$-tuples of vectors in $E$ where the first $(d+1)$ vectors are affinely independent. We shall refer to such $(k+1)$-tuples as non-singular. 

\vskip.125in 

We say that two non-singular $(k+1)$-tuples $x^1, x^2, \dots, x^{k+1}$ and $y^1, y^2, \dots, y^{j+1}$ are congruent if there there exists a rotation $\theta$ and a translation $\tau$ such that 
$$y^j=\theta x^j+ \tau.$$ Let $M_d(k)(\mathbb{R}^d)$ denote the set of the resulting equivalence classes. Let $M_d(k)(E)$ denote the set of resulting equivalence classes where the vectors are restricted to a finite point set $E$. 

\begin{theorem} \label{Thm: PlaneCongSharp}
Let $E$ be a finite point set in $\mathbb{R}^2$. Then
$$ |M_2(k)(E)| \gtrapprox |E|^k,$$ where here and throughout, $X \lessapprox Y$ with the controlling parameter $R$ means that given $\epsilon>0$ there exists $C_{\epsilon}>0$ such that $X \leq C_{\epsilon}R^{\epsilon}Y$. 

\vskip.125in 

\noindent Moreover, the lower bound is, in general, best possible. 
\end{theorem}

\vskip.125in 

We could state a higher dimensional version of Theorem \ref{Thm: PlaneCongSharp}, but it would not be sharp because our argument relies to a significant extent on the case $k=1$ where the needed bound is only known in two dimensions.

We now deal with point configuration where distances between some pairs of points are specified and others are not. An interesting and deceptively looking example is provided by the hinge. More precisely, it is reasonable to ask if $E$ is a finite subset of ${\Bbb R}^2$, whether 
\begin{equation} \label{hingequestion} |\{(|x-y|,|x-z|): x,y,z \in E \}| \gtrapprox {|E|}^2. \end{equation}

\begin{center}
\begin{figure}[ht] 
\begin{tikzpicture}[scale=2]
\coordinate (A) at (0cm,0cm);
\draw (A) node[anchor=east] {$y$};
\coordinate (B) at (1.4cm,1cm);
\draw (B) node[anchor=south] {$x$};
\coordinate (C) at (2cm,0cm);
\draw (C) node[anchor=west] {$z$};
\coordinate (D) at (0.5cm,-0.7cm);
\draw[thick, line cap=round] (A) -- (B);
\draw[thick, line cap=round] (B) -- (C);
\end{tikzpicture}
\caption{The hinge}\label{figure2} 
\end{figure}
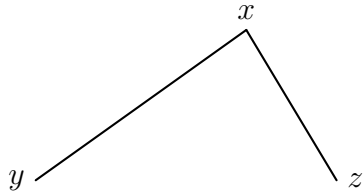
\end{center}

We can gain a non-optimal bound on the hinge (as as we will later state, general non-rigid configurations) using pinned distance bounds. One defines pinned distance in the plane as
$$ \Delta_x(E) = \{ |x-y| : y \in E\}, $$
for a pin $x \in E$.
The best known pinned result for the Erd\H os distance problem is due to Katz and Tardos (\cite{KT04}). They proved that there exists $x \in E$ such that 
\begin{equation}
\Delta_x(E)| \gtrapprox {|E|}^{(48-14e)/(55-16e)}. \label{Eq: PinnedPlaneBound}
\end{equation}

It follows that 
$$  |\{(|x-y|,|x-z|): x,y,z \in E \}| \gtrapprox {|E|}^{2(48-14e)/(55-16e)}.$$ 

Taking $E={\Bbb Z}^2 \cap {\left[ 0,\sqrt{n} \ \right]}^2$ shows that the estimate (\ref{hingequestion}) would be best possible. While this question looks like a natural variant of the Erd\H os distance conjecture, it appears to be very difficult. In order to study configurations of this type, we need to build a geometric mechanism for point configuration with distance relations encoded by combinatorial graphs. This is where we now turn our attention. The main theorem resulting from this machinery is Theorem \ref{Thm: PlanarGraphRigidSharp} below. 

\vskip.25in 

\subsection{Graph Rigidity} To gain sharp bounds on the size of individual congruence classes more structure on these finite point configurations which allow us to prove more sharp bounds in the plane we shall encode finite point frameworks using combinatorial graphs. Let $k\geq 1$ and let $K_{k+1}$ denote the complete graph with vertex set $\{1,\dots, k+1\}$ and edge set ordered lexicographically. Let $G_{k+1,m}$ be a subgraph of $K_{k+1}$ with $k+1$ vertices and $m$ edges inheriting the order. 

\begin{definition}
A \emph{$(k+1)$-tuple $\bm x$ in $\mathbb R^d$} is a tuple
$$\bm x = (x^1, x^2, \dots, x^{k+1}), \ x^j \in \mathbb R^d\ .$$
\end{definition}

\begin{definition} A \emph{framework of $G_{k+1,m}$ in ${\mathbb R}^d$} is a pair $(G_{k+1,m}, \bm x)$, where $\bm x$ is a $(k+1)$-tuple in $\mathbb R^d$.\end{definition} 

A convenient way to specify distances is through the distance function which we now define.
\begin{definition} \label{defdistancefunction} Given a graph $G_{k+1,m}$ we define the \emph{distance function} $f_{G_{k+1,m}}(\bm x)$ on $\bm x = (x^1, \dots, x^{k+1}) \in \mathbb R^{d(k+1)}$ by
$$f_{G_{k+1,m}}(\bm x) = \left(|x^i - x^j|\right)_{ij\in G_{k+1,m}}\ .$$
We also define the \emph{distance-squared function} $F_{G_{k+1,m}}(\bm x)$ by
$$F_{G_{k+1,m}}(\bm x) = \left(|x^i - x^j|^2\right)_{ij\in G_{k+1,m}}\ .$$
\end{definition}

\begin{definition}[Graph Distances]\label{graphdists} The value $f_{G_{k+1,m}}(\bm x)$ is called the \emph{$G_{k+1,m}$-distance of $\bm x$}. When we restrict our domain to some set $\mathfrak X\subseteq \mathbb R^{d(k+1)}$, we call $f_{G_{k+1,m}}(\bm x)$ a \emph{$G_{k+1,m}$-distance on $\mathfrak X$} and we say that $\bm x$ is a realization of this distance in $\mathfrak X$. The set of $G_{k+1,m}$-distances on $\mathfrak X$ is $f_{G_{k+1,m}}(\mathfrak X)$ and we denote it by $\Delta(G_{k+1,m}, \mathfrak X)$. \end{definition}

\begin{remark} The distance set $\Delta(G_{k+1,m}, \mathfrak X)$ depends on the numbering of the vertices and the order of the edges. Whereas the order of the edges is superficial, inducing only a permutation in the components of the $G_{k+1,m}$-distances, the numbering of the vertices can significantly change the $G_{k+1,m}$-distance set. Consider $\mathfrak X = \{\bm x_0\}\times \mathbb R^d \times \cdots \times \mathbb R^d$ and a graph $G=G'\cup G''\cup\{e\}$ where $e$ is a bridge between $G'$ and $G''$. Then if we number the vertices of $G'$ followed by those of $G''$, we essentially capture $G''$-distances only, whereas if we reverse the numbering order of the vertices of $G$ we will capture $G'$-distances only. In the rest of this paper we take $\mathfrak X = E^{k+1}$ for some finite $E\subset \mathbb R^d$, so that the numbering of the vertices becomes superficial as well. In particular, the size of the $G_{k+1,m}$-distance set is independent of the vertex numbering and edge order. \end{remark} 

We consider the following conjecture,

\begin{conjecture}\label{Conj: GraphDistanceSharp} Let $E$ be a finite set in the plane of size $n$ and $G_{k+1,m}$ be a connected graph on $k+1$ vertices having $m$ edges. Then, $|\Delta(G_{k+1,m}, E^{k+1})| \gtrapprox n^k$. \end{conjecture}

\begin{theorem}\label{Thm: ConjIsSharp}
Conjecture \ref{Conj: GraphDistanceSharp} is sharp.
\end{theorem}

Our main results here concern the size of the set $\Delta(G_{k+1,m}, E^{k+1})$. An important role is played by properties of the graph $G_{k+1,m}$. In particular it is essential whether the graph is rigid or not. 

The key heuristic notion of this paper is that a graph $G_{k+1,m}$ is \emph{rigid in $\mathbb R^d$} if once the $m$ quantities $t_{ij}$ in
$$|x^i - x^j| = t_{ij}, \qquad ij\in G_{k+1,m}$$
are specified, the other distances $|x^i - x^j|$ for $ij\not\in G_{k+1,m}$ can only take finitely many values as the frameworks $(G_{k+1,m}, \bm x)$ vary over the set of non-degenerate frameworks (see generic frameworks below for a formal definition of this non-degeneracy).

For technical reasons, we use a more precise and flexible notion of rigidity described below. A simple example that illustrates the technical obstacles one must contend with is the following. Consider a quadrilateral in the plane with side-lengths $1,1,1,3$. This configuration is perfectly rigid in the heuristic sense, but it is not {\it minimally infinitesimally rigid}, as the reader will see, roughly because the rigidity in this case is not stable under small perturbations. 

\vskip.125in 

We now turn to precise definition. 
 
\begin{definition}
An \emph{infinitesimal motion $\bm u = (u^1, \dots, u^{k+1})$ in $\mathbb R^d$ of $G_{k+1,m}$ at $\bm x$} is a $(k+1)$-tuple $\bm u$ of vectors $u^j\in \mathbb R^d$ such that
$$DF_{G_{k+1,m}}(\bm x)\cdot \bm u = 0\ .$$
The set of infinitesimal motions in $\mathbb R^d$ of $G_{k+1,m}$ at $\bm x$ is the kernel of $DF_{G_{k+1,m}}(\bm x)$.
Let us denote by $\mathcal V(G_{k+1,m}, \bm x)$ the set of infinitesimal motions in $\mathbb R^d$ of $G_{k+1,m}$ at $\bm x$. Let $\mathcal D(\bm x)$ be the set of infinitesimal motions in $\mathbb R^d$ of $K_{k+1}$ at $\bm x$.
\end{definition}
\begin{remark}
It is evident that $\mathcal D(\bm x) \subseteq \mathcal V(G_{k+1,m}, \bm x)$ since the system of equations $DF_{G_{k+1,m}}(\bm x)\cdot \bm u = 0$ is included in $DF_{K_{k+1}}(\bm x)\cdot \bm u = 0$.
\end{remark}
\begin{definition}\label{infrigidframework}
A framework $(G_{k+1,m}, \bm x)$ is called \emph{infinitesimally rigid in $\mathbb R^d$} when \linebreak
$\mathcal V(G_{k+1,m}, \bm x) = \mathcal D(\bm x)$.
\end{definition}

It is unnecessarily restrictive to require of a graph to have all its frameworks be infinitesimally rigid. We shall only require it of a certain family of frameworks which we call generic frameworks. Below we define the set of generic tuples as the complement of the zero set of a certain polynomial. This notion is independent of the graph $G_{k+1,m}$, depending only on the dimension $d$ and the number of vertices $k+1$. \\

We also define the notion of independence for subsets of the edge set of $K_{k+1}$ and of maximal independence for subsets of the edge set of $G_{k+1,m}$. 

Let us use the following notation for our matrices: If $a_{ij}$ is a matrix, $(i,j)\in I\times J$, then for $B\subseteq I, C\subseteq J$, we defined $a_{B,C}$ to be the submatrix $a_{ij}$ with $(i,j)\in B\times C$. 

\begin{definition}
We say that $\bm x \in \mathbb R^{d(k+1)}$ is a \emph{regular tuple of $F_{G_{k+1,m}}$} if $\operatorname{rank}DF_{G_{k+1,m}}$ attains its global maximum at $\bm x$. A framework $(G_{k+1,m}, \bm x)$ is a \emph{regular framework} if $\bm x$ is a regular tuple of $F_{G_{k+1,m}}$.
\end{definition}

\begin{definition}
A subset $H$ of the edge set of $K_{k+1}$ is called \emph{independent in $\mathbb R^d$ with respect to $\bm x_0 \in \mathbb R^{d(k+1)}$} if the row vectors of $DF_{K_{k+1}}(\bm x_0)$ corresponding to $H$ are linearly independent. We call $H$ \emph{independent in $\mathbb R^d$} if there exists some $\bm x_0$ so that $H$ is independent with respect to $\bm x_0$, and $\bm x_0$ is said to be a \emph{witness to the independence of $H$}. We also call $H$ a \emph{maximally independent (in $\mathbb R^d$) subset of edges of $G_{k+1,m}$} when it is independent and it is not contained in a larger independent edge set of $G_{k+1,m}$.
\end{definition}

\begin{definition}\label{defgenericpts}
For any nonempty independent in $\mathbb R^d$ set $H$ of edges of $K_{k+1}$ we define the polynomial $P_{H}(\bm x)$ to be the sum of squares of $|H|\times |H|$-minors of the submatrix of rows of $DF_{K_{k+1}}$ corresponding to edges of $H$.  Thus,
\begin{align*}
P_H(\bm x) = \sum_{\substack{A\subset\{1,\dots,d(k+1)\} \\ |A| = |H|}} \left|\det (DF_{K_{k+1}}(\bm x)_{H,A})\right|^2\ .
\end{align*}
Let $X_H$ denote the zero set of $P_H$.

We define the set of \emph{generic tuples of $\mathbb R^d$} to be the complement of the zero set $X$ of the polynomial $P(\bm x)$ defined by
\begin{align*}
P(\bm x) = \prod_{H\text{ independent}} P_H(\bm x)\ .
\end{align*}
We call $X$ the set of \emph{critical tuples of $\mathbb R^d$}.
\end{definition}
\begin{remark}
We have $X = \cup_H X_H$ where the union is taken over all the edge sets $H$ which are independent and the generic tuples are then equal to $\mathbb R^{d(k+1)}\setminus X$. Moreover, if a set $H$ of edges is independent then by Definition \ref{defgenericpts} it is generically independent, i.e. independent with respect to any generic $\bm x$. In fact, the set of generic tuples is precisely the set of tuples that simultaneously witness the independence of every independent edge set.
\end{remark}
\begin{remark}
The polynomial $P(\bm x)$ is nontrivial because every $P_H$ is nontrivial since there is at least one witness $\bm x_H$ for the independence $H$, which means that $P_H(\bm x_H) \not = 0$. Thus $X$ is a proper algebraic variety of dimension
\begin{align}
\dim X \leq d(k+1) - 1 \label{critdim}\ .
\end{align}
\end{remark}
\begin{remark}\label{genericimpliesregular}
It is immediate from the definitions that generic tuples are regular tuples. The other implication does not hold in general.
\end{remark}
\begin{definition}
A framework $(G_{k+1,m}, \bm x)$ is called \emph{generic in $\mathbb R^d$} if $\bm x$ is a generic tuple in $\mathbb R^d$ and it is called \emph{critical in $\mathbb R^d$} if $\bm x$ is a critical tuple in $\mathbb R^d$.
\end{definition}

Thus we can now complete our formal nation of graph rigidity rigid.

\begin{definition}
A graph $G_{k+1,m}$ is called infinitesimally rigid in $\mathbb{R}^d$ if all its generic frameworks are infinitesimally rigid. It is called minimally infinitesimally rigid in $\mathbb{R}^d$ if it is infinitesimally rigid and no proper subgraph (on the same vertex set) is infinitesimally rigid.
\end{definition}

Using the notion of minimally infinitesimally rigid we can gain sharp results for many graphs. Formally,

\begin{theorem} \label{Thm: PlanarGraphRigidSharp} If $G_{k+1,m}$  be a minimally infinitesimally rigid connected graph on $k+1$ vertices having $m$ edges and $E$ be a finite set in the plane of size $n$. Then, $|\Delta(G_{k+1,m}, E^{k+1})| \gtrapprox n^k$.\end{theorem}
\vspace{0cm}

\vskip.125in 

\subsection{Erd\H os Pinned Distance Conjecture}

Our final result allows us to drop the condition that our graph $G_{k+1,m}$ need be rigid. We do this by evoking Erd\H os' pinned distance conjecture. As states earlier the current best know result is \ref{Eq: PinnedPlaneBound} due to Katz and Tardos (\cite{KT04}) The established conjecture is the following:

\begin{conjecture}
For a finite point set $E$ in $\mathbb{R}^d$ there is a point $y$ in $E$ such that $|\Delta_y(E)| \approx |E|^{\frac{2}{d}}$.
\end{conjecture}

It is clear, by considering the integer lattice, that the above conjecture is the best one can hope for. With our current technology we are far from this sharp result, however will demonstrate the general graph distances result is a closely linked, though weaker, result.

\begin{conjecture}\label{Conj: GraphDistanceSharp} Let $E$ be a finite set in the plane of size $n$ and $G_{k+1,m}$ be a connected graph on $k+1$ vertices having $m$ edges. Then, $|\Delta(G_{k+1,m}, E^{k+1})| \gtrapprox n^k$. \end{conjecture} 

\begin{theorem} \label{Thm: GenGraphErdosPin} The conjecture above holds for any $G_{k+1,m}$ if the pinned distance conjecture is assumed. \end{theorem}

\section{Proof of Theorem \ref{Thm: PlaneCongSharp}} \label{Sec: PlaneCongSharpProof}

We begin by deriving the properties of $M_d(k)(\mathbb{R}^d)$ that we shall need in the proof. 

\vskip.125in 

\subsection{Congruence Classes of $k+1$ tuples}

In this section we build a congruence relation - using the action of the orthogonal group - to provide a more general class of configuration where we can gain sharp results on distance tuples.
~\\
A $(k+1)$-point configuration in $\mathbb{R}^d$ is given by an arbitrary choice of point in $\mathbb{R}^{d(k+1)}$. We label this configuration as 
$(v_0, \dots, v_k)$ with $v_j \in \mathbb{R}^d$ initially.

Recall, $k \geq d$ is assumed and we say that the configuration above is non-singular if its first $(d+1)$ vectors $\{ v_0, \dots v_d \}$ are affinely independent. 

We denote the space of these non-singular $(k+1)$-point configurations in $\mathbb{R}^d$ by $N_d(k)$.

~\\
{\raggedleft \bf Step 1: Passage to origin pinned configurations.}
~\\

Given a non-singular configuration $(v_0, \dots, v_k)$ we define the associated origin-pinned configuration as $(u_1, \dots u_k)$ 
where $u_j = v_j - v_0$ for all $1 \leq j \leq k$. The first $d$-resulting vectors of this process form an invertible matrix whose columns are $u_1, \dots, u_d$ as these vectors were required to be linearly independent and hence are a basis of $\mathbb{R}^d$. Thus we will write the associated origin-pinned configuration as $(\mathbb{A}, u_{d+1}, \dots, u_k)$.

The space for non-singular origin-pinned configurations is hence identified as 
$$GL_d(\mathbb{R}) \times \mathbb{R}^{d(k-d)}.$$

So, we have a map 
$$\pi: N_d(k) \to GL_d(\mathbb{R}) \times \mathbb{R}^{d(k-d)}$$ 
given by 
$$\pi(v_0, \dots, v_k) = (\mathbb{A}, v_{d+1}-v_0, \dots, v_k-v_0).$$ 
This map is equivalent to passage to translation classes of non-singular configurations of $(k+1)$ points in $\mathbb{R}^d$.

~\\
{\raggedleft \bf Step 2: Analysis of $O(d)$ action on pinned configurations and ``moving frames."}
~\\

Congruence classes of such pinned configurations are given by $O(d)$-orbits of the action given by $B \in O(d)$ acts on $(A, u_{d+1}, \dots, u_k)$ by sending it to $(BA, Bu_{d+1}, \dots, Bu_k)$.

This action is complicated by the fact that $O(d)$ acts on both the matrix $A$ and the remaining vectors. To simplify future formulas we fix this by using a method of ``moving frames."

All this means is that as the columns of $A$ are a basis of $\mathbb{R}^d$, we may expand each $u_j$ when $j > d$ as a linear combination of $u_1, \dots u_d$. 
If $u_j = \sum_{k=1}^d c_{jk} u_k$ we will define $c_j=(c_{j1}, \dots, c_{jd})^T$. Equivalently $Ac_j=u_j$, note as $A$ depends on the first $d$ vectors, this is a variable change of basis, i.e. a ``moving frame."

Notice now when $B \in O(d)$ acts, $Bu_j=\sum_{k=1}^d c_{jk} Bu_k$, or equivalently $BAc_j = Bu_j$, and so the $c_j$ vectors themselves are unchanged by the $O(d)$-action.

In other words if we reencode pinned configurations as $(u_1, \dots, u_d, c_{d+1}, \dots, c_{k})$ then the $O(d)$ action only acts on the first $d$-coordinates and leaves the remaining coordinates unchanged. Thus the action becomes 
$$B \cdot (A, c_{d+1}, \dots, c_k) = (BA, c_{d+1}, \dots, c_k),$$ 
so now $O(d)$ will only act on the matrix slot in this coordinate system.

To summarize, we will now use this ``moving frames" coordinate system, and thus an origin pinned configuration is given by $(A, c_{d+1}, \dots, c_k) \in GL_d(\mathbb{R}) \times \mathbb{R}^{d(k-d)}$ where $Ac_j=u_j$ relates the original vectors to these new $c$-vectors. 

\vskip.125in 

{\raggedleft \bf Step 3: Quotienting $O(d)$-action.} Using the moving frame coordinate system, nonsingular origin-pinned configurations of $(k+1)$ points in $\mathbb{R}^d$ is the space $GL_d(\mathbb{R}) \times \mathbb{R}^{d(k-d)}$. The action of $O(d)$ is given by $B \cdot (A, c_{d+1}, \dots, c_k) = (BA, c_{d+1}, \dots, c_k)$ so the final space for nonsingular congruence classes of configurations of $(k+1)$-points in $\mathbb{R}^d$, which we will call $M_d(k)(\mathbb{R}^d)$, is given by 
$$M_d(k)(\mathbb{R}^d) = (O(d) \backslash GL_d(\mathbb{R})) \times \mathbb{R}^{d(k-d)}.$$
Where this is the quotient of the left-action of $O(d)$.
To make this more explicit, we recall the $LU$ or $UL$-decomposition of nonsingular matrices that comes from the Gram-Schmidt process.
Any $A \in GL_d(\mathbb{R})$ can be written $A=BC$ for unique $B \in O(d), C \in L$ where $L$ is the Lie group of upper triangular matrices with positive real entries on the diagonal.

This means as manifolds (but not as groups) $GL_d(\mathbb{R})$ is diffeomorphic to $O(d) \times L$ where the left action of $O(d)$ on $GL_d(\mathbb{R})$ translates to an action on 
$O(d) \times L$ where $O(d)$ acts only on the left factor by left translation. Thus $O(d) \backslash GL_d(\mathbb{R})$ is naturally diffeomorphic to the Lie group $L$.

Putting this all together we have:

\begin{summary}
Let $(v_0, v_1, \dots v_k)$ be a $(k+1)$-configuration in $\mathbb{R}^d$ with $k \geq d$ and the first $d+1$ vectors affinely independent. Then we define $u_j=v_j - v_0, 1 \leq j \leq k$ and 
make a matrix $A \in GL_d$ with $u_1, \dots, u_d$ as column vectors. The data $(A, u_{d+1}, \dots, u_k) \in GL_d \times \mathbb{R}^{d(k-d)}$ encodes the origin-pinned configurations or equivalently the translation classes of nonsingular configurations. 

We then change coordinates to a moving frame coordinate system $\mathbb{A}c_j = u_j$ for $d+1 \leq j \leq k$.
The data $(A, c_{d+1}, \dots, c_k)$ also encodes pinned configurations, but now the $O(d)$-action is only on the $A$-coordinate.

Finally we mod the $O(d)$ action to get the space of congruence classes of nonsingular $(k+1)$ point configurations in $\mathbb{R}^d$, which is called $M_d(k)(\mathbb{R}^d)$.

$$
M_d(k)(\mathbb{R}^d) = L \times \mathbb{R}^{d(k-d)},
$$

where the final data is $(C, c_{d+1}, \dots c_{k})$ with $A=BC, B \in O(d), C \in L$ the $UL$-decomposition of $A$. 
$L$ is the Lie group of upper triangular matrices with positive real entries on the diagonal.
\end{summary}

\vspace{0.2cm}

\subsection{Proof of Theorem \ref{Thm: PlaneCongSharp}}~\\

To prove Theorem \ref{Thm: PlaneCongSharp} we shall use the following famous Theorem of Guth and Katz that resolved the Erd\H os distance problem in the plane (see \cite{GK15}).

\begin{theorem}[Guth-Katz]
Suppose that $E$ is a finite point set in $\mathbb{R}^2$ and let $\theta$ be an orthogonal transformation on the plane, with $v_\theta(t) = \{(x,x') \in E^2 : x- \theta x' = t\}$. Then,
$$ \sum_{t\in \mathbb{R}^2}\sum_{\theta \in O(2)} v^2_\theta(t) \lesssim |E|^3\log(|E|).$$
\end{theorem}

\begin{proof}
For $\mathcal{S}$ in $M_2(k)(E)$, let $\lambda(\mathcal{S})$ be the orbit of $\mathcal{S}$ under the $O(2)$ action. Then,
\begin{align*}
|E|^{2(k+1)} &= \left( \sum_{\mathcal{S} \in M_2(k)(E)}\lambda(C) \right)^2 \leq |M_2(k)(E)| \sum_{\mathcal{S} \in M}\lambda^2(\mathcal{S}) \\
					 &= \sum_{\tau \in \mathbb{R}^2} \sum_{\theta \in O(2)} v_\theta^{(k+1)}(\tau) \lesssim |E|^{(k+2)}\log(|E|)
\end{align*}
\end{proof}

\section{Proof of Theorem \ref{Thm: PlanarGraphRigidSharp}}

The proof of Theorem \ref{Thm: PlanarGraphRigidSharp} follows from the fact that there can only be a constant number, dependent only on the number of points $k$ and the dimension $d$, of congruences associated to a minimally infinitesimally rigid graph. To prove this result we will follow the outline of \cite{CIMP17}.

For a tuple $\textbf{x}$ in $\mathbb{R}^{d(k+1)}$ it is useful to define the following pre-image of $f_{G_{k+1,m}}(\textbf{x})$

\begin{equation}\label{Eqn: PreImFramework}
N_{\bf x} = \{\textbf{y} \in \mathbb{R}^{d(k+1)}: f_{G_{k+1,m}}(\textbf{y})=f_{G_{k+1,m}}(\textbf{x})\}.
\end{equation}

\begin{proposition}[Section 3.4, \cite{CIMP17}]\label{Prop: FiniteConnComp}
Suppose $G_{k+1,m}$ a minimally infinitesimally rigid graph, $\textbf{x}$ any tuple (regular is not necessary here), let $b_0(N_x/\sim)$ denote the number of connected components of $N_\textbf{x}$ under the congruence relation given by the $O(d)$ action. Then $b_0(N_\textbf{x} /\sim) \leq C_{k,d}$, for some number $C_{k,d}>0$, depending only on the dimension $d$ and the number of points $k$.
\end{proposition}

\begin{proposition}[Proposition 4.11, \cite{CIMP17}] \label{Prop: ConnectedThenMotion}
Suppose $G_{k+1,m}$ a minimally infinitesimally rigid graph, $\textbf{x}$ any regular tuple of $\G$. If $\textbf{y}$ and $\textbf{z}$ are in the same connected component of $N_\textbf{x}$ then there is some $\theta$ in $ISO(\mathbb{R}^d)$ such that $\textbf{y}=\theta \textbf{z}$.
\end{proposition}

We can combine the above two results to give us the Theorem.

Let us first define the following set

$$ v(t) = \{ \textbf{x} \in E^{k+1} : f_{\G}(\textbf{x}) = t\}.$$

Using Proposition \ref{Prop: FiniteConnComp} that we can divide $v(t)$ into a finite union of connected components $\tilde{v}_i$. Thus 
$$v(t) = \bigcup_{i=1}^{C_{k,d}} \tilde{v}_i(t),$$
where some $\tilde{v}_i(t)$ may be empty. Letting $\tilde{v}_0(t)$ be the largest of these connected components we have the following estimate.

\begin{align*}
|E|^{2(k+1)} &= \left( \sum_{t \in \Delta(\G)(E)} v(t) \right)^2 =  \left( \sum_{t \in \Delta} \sum_{i=1}^{C_{k,d}} \tilde{v}_i(t) \right)^2 \\
					&\leq C_{k,d}^2 \left( \sum_{t \in \Delta} \tilde{v}_0(t) \right)^2 \\
					&\leq C_{k,d}^2 |\Delta(\G)(E)| \sum \tilde{v}_0^2(t).
\end{align*}

Thus to prove the result it suffices to prove the following bound,

$$ \sum_t \tilde{v}_0^2(t) \lesssim |E|^{(k+2)}\log(|E|).$$

To do this we need to use Proposition \ref{Prop: ConnectedThenMotion}. Note that,

$$\sum_t \tilde{v}_0^2(t) = |\{ (\textbf{x},\textbf{y}) | f_{\G}(\textbf{x}) = f_{\G}(\textbf{y})~\&~ \textbf{x},\textbf{y} \text{ in same max. conn. comp. of } f^{-1}_{\G}(\Delta)\}|.$$

By Proposition \ref{Prop: ConnectedThenMotion} we have that $x$ and $y$ being in the same connected component of $f^{-1}_{\G}(\Delta(\G)(E))$ means there is a rigid motion $\theta$ such that $\textbf{x}=\theta \textbf{y}$. Recalling that these are frameworks we have that $(x^1, \ldots, x^{(k+1)}) = (\theta y^1, \ldots, \theta y^{(k+1)})$. Using that $f_{\G}(\textbf{x}) = f_{\G}(\theta\textbf{y})$ so then if $ij$ and edge in $\G$ we have $x^i-\theta y^i = x^j - \theta y^j = \tau$ where $\tau$ is uniform over the tuple pair $(\textbf{x}, \textbf{y})$.

So if we define

$$ v_\theta(\tau) = \{ (x,y) \in E^2 : x-\theta y = \tau\},$$

we have the following result,

$$ \sum_t \tilde{v}_0^2(t) \leq \sum_{\tau \in \mathbb{R}^d} \sum_{\theta \in ISO(\mathbb{R}^d)} v^{(k+1)}_\theta(\tau).$$

Here we don't necessarily have equality here as there may be elements counted in the right hand side that are outside the maximal connected component of  $f^{-1}_{\G}(\Delta(\G)(E))$. But certainly all pairs from the maximal connected component are counted\footnote{In fact the RHS counts all pairs from each connected component, but not cross pairs. However we have to reduce to one connected component to pass through the C-S step above.}. This bound suffices for our purposes and will in fact produce a sharp result. To conclude we note the following trivial bound,

$$ |v_\theta(\tau)| \leq |E|,$$

which follows from the fact that the second coordinate is entirely dependent on the choice of the first (once $\theta$ and $\tau$ are fixed).

Until this stage the calculation works in any dimension $d$, however to conclude we are going to apply the Guth-Katz result that lead to the resolution of the Erd\H os distance problem. This requires that we operate in dimension 2 only. When $d=2$ we have

\begin{align*}
\sum_t \tilde{v}_0^2(t) &\leq\sum_{\tau \in \mathbb{R}^2} \sum_{\theta \in ISO(\mathbb{R}^2)} v^{(k+1)}_\theta(\tau) \\
										&\leq |E|^{(k-1)} \sum_{\tau \in \mathbb{R}^2} \sum_{\theta \in ISO(\mathbb{R}^2)} v^{2}_\theta(\tau) \\
										&\lesssim |E|^{(k-1)} \cdot |E|^3 \log(|E|) = |E|^{(k+2)}\log(|E|).
\end{align*}

Where the final estimate deploys the Guth-Katz result. This was the bound we required, thus we have for a minimally infinitesimal graph $\G$ that
$$ |\Delta(\G)(E)| \gtrapprox |E|^k. $$

\section{Proof of Theorem \ref{Thm: GenGraphErdosPin}}

\noindent Recall we define the pinned distance set as
$$ \Delta_x(E) = \{ |x-y| : y \in E\}, $$
for a pin $x \in E$. We call $|\Delta_y(E)|$ the pin-richness of $x$ (in $E$) and a set $A$ a $r$-rich pin set if every point in $A$ has pin-richness at least $r$. Recall that the Erd\H os pinned distance conjecture states:

\begin{conjecture}
For a finite point set $E$ in $\mathbb{R}^d$ there is a point $y$ in $E$ such that $|\Delta_y(E)| \approx |E|^{\frac{2}{d}}$.
\end{conjecture}

The first part of our prove is to show that if the above conjecture hold then we have many rich pins. We can then use these rich pins as the vertices for our distance graphs, where their richness allows us to construct sufficiently many variations of graph-distance tuples.

\begin{lemma}\label{Lem: ManyRichPins}
Suppose the Erd\H os pinned-distance conjecture is satisfied for a point set $E$. Then there are $\sim |E|$ points $x$ in $E$ such that $|\Delta_{x}(E)| \approx |E|$.
\begin{proof}
To see this we use Erd\H os' pinned-distance conjecture to find a pin $x_0$ such that $|\Delta_{x_0}(E)| \approx |E|$. We then remove this point from $E$ to gain a modified $E_0$. We then apply the conjecture to $E_0$ to gain some $x_1$ which is a pin of richness $\approx |E|$. We repeat the process $\frac{|E|}{2}$ times gaining a sufficiently rich pin each time. Thus we have $|E|/2$ pins with pin richness between $|E|$ and $|E|/2$ as claimed. 
\end{proof}
\end{lemma}

To finish the proof we count the number of possible distance drawing using the rich-pin subset of $E$. Notice that for any graph drawing once we have determined the position of the vertices we have no freedom left to select any other edges. Thus we naturally use spanning trees to determine the number of ways we have of drawing the graph. It is clear that for a graph on $k$ vertices that the number of edges in the spanning tree will be $k-1$. As we have that $k <<|E|$ (in particular $k<<\frac{|E|}{2}$) then we can choose our edges essentially independently from the set of rich distances.

Thus according to Lemma \ref{Lem: ManyRichPins} the total number choices for each edge in the spanning tree is $\approx |E|^{\frac{2}{d}}$. As we have $k-1$ such choices and our choices are independent, we have a total number of choices is

$$ \gtrapprox \left(|E|^{\frac{2}{d}} \right)^{(k-1)} = |E|^{\frac{2(k-1)}{d}}.$$

We note that this is clearly sharp as the grid in $\mathbb{R}^d$ satisfies the Erd\H os distance problem criterion, in that each point has $\sim |E|^{\frac{2}{d}}$ unique distances in its pinned distance set.

\bigskip

\end{document}